\documentclass{amsart}
\usepackage{macros}
\standardsettings

\newcommand{\brmatrix}[1]{\left[\begin{matrix} #1 \end{matrix}\right]}

\colorcommentstrue

\draftfalse

\begin{document}
\title{Decaying and non-decaying badly approximable numbers}

\authorryan\authorlior\authordavid

\begin{Abstract}
We call a badly approximable number \emph{decaying} if, roughly, the Lagrange constants of integer multiples of that number decay as fast as possible. In this terminology, a question of Y. Bugeaud ('15) asks to find the Hausdorff dimension of the set of decaying badly approximable numbers, and also of the set of badly approximable numbers which are not decaying. We answer both questions, showing that the Hausdorff dimensions of both sets are equal to one. Part of our proof utilizes a game which combines the Banach--Mazur game and Schmidt's game, first introduced in Fishman, Reams, and Simmons (preprint '15).
\end{Abstract}
\maketitle

\section{Introduction}
\noindent Fix $d\in\N$. For each $\xx\in\R^d$, write
\begin{align*}
L(\xx) &= \liminf_{q \to \infty} q^{1/d} \dist(q\xx,\Z^d) \\
&= \sup\{c \ge 0 \colon \|\xx - \pp/q\| \ge c/q^{1 + 1/d}\text{ for all but finitely many }(\pp,q) \in \Z^d\times \N\},
\end{align*}
where $\|\cdot\|$ denotes some norm on $\R^d$ (e.g. the max norm). The point $\xx$ is called \emph{well approximable} if $L(\xx) = 0$, and \emph{badly approximable} if $L(\xx) > 0$. If $\xx$ is badly approximable and $i/j\in\Q$, we can get a trivial bound on $L\big(\frac ij \xx\big)$ in terms of $L(\xx)$: if $L(\xx) > c$, then for all but finitely many $(\pp,q) \in \Z^d\times \N$
\[
\left\| \frac ij \xx - \frac \pp q\right\| = \frac ij\left\|\xx - \frac{j\pp}{i q}\right\| \ge \frac ij \frac{c}{(iq)^{1 + 1/d}} = \frac{c}{i^{1/d} j} \frac{1}{q^2}
\]
and thus
\begin{equation}
\label{crude}
L\left(\frac ij \xx\right) \ge \frac{1}{i^{1/d} j}L(\xx).
\end{equation}
The special case $j = 1$ gives the bound
\begin{equation}
\label{crude2}
L(i\xx) \geq L(\xx)/i^{1/d}.
\end{equation}
Recently, Y. Bugeaud \cite{Bugeaud3} has investigated the question of whether the bound \eqref{crude2} is ``optimal''.  Precisely, call a badly approximable point $\xx$ \emph{decaying} if there exists $C > 0$ such that for all $i\in\N$,
\[
L(i\xx) \leq C/i^{1/d}.
\]
In this terminology, the main result of \cite{Bugeaud3} states that certain algebraic badly approximable points are decaying. Bugeaud also asks the following question:

\begin{question}[{\cite[Problem 4.4]{Bugeaud3}}]
\label{questionbugeaud}
What is the Hausdorff dimension of the set of decaying badly approximable\Footnote{The condition ``badly approximable'' has been omitted in the statement of \cite[Problem 4.4]{Bugeaud3}, but the question obviously makes no sense without it.} points? of the set of badly approximable points which are not decaying?
\end{question}

The latter question has been recently answered in dimension 1 by a preprint of D. Badziahin and S. Harrap \cite[Theorem 15]{BadziahinHarrap}.\Footnote{In the current version of their paper they do not mention this connection; we informed them of it by private communication while writing this paper.} In this paper, we give a different proof of Badziahin--Harrap's result, which is valid in higher dimensions as well. Our proof shows that the set of badly approximable points which are not decaying is hyperplane winning, which shows that its intersection with certain ``nice'' fractals has full Hausdorff dimension; see \6\ref{sectiongames} for more details.

We also answer Bugeaud's first question in dimension 1, showing that the set of decaying badly approximable numbers has full dimension. Our proof uses the so-called Banach--Mazur--Schmidt game recently introduced in \cite{FRS}.

{\bf Acknowledgements.} The second-named author was supported in part by the Simons Foundation grant \#245708. The third-named author was supported in part by the EPSRC Programme Grant EP/J018260/1. We wish to thank Jayadev Athreya for introducing us to a special case of the theorems proved in this paper which motivated this work. We thank the anonymous referee for helpful comments.

\subsection{Main results}
In our first main theorem, we find a full dimension set of badly approximable points which are not only non-decaying, but non-decaying when multiplied by any chosen sequence of rational numbers, with as little decay in the Lagrange values as desired. (The full dimension set depends on the chosen sequence and decay rate.) Precisely:

\begin{theorem}
\label{theoremslow}
Let $\frac{i_1}{j_1},\frac{i_2}{j_2},\ldots \in \Q$ be a sequence of distinct rational numbers, and let $g:\N\to(0,\infty)$ satisfy $g(k)\to\infty$ as $k\to\infty$. Then the set
\begin{equation}
\label{slow}
\left\{\xx\in\BA_d : \limsup_{k\to\infty} g(k) L\left(\frac{i_k}{j_k}\xx\right) = \infty \right\}
\end{equation}
is hyperplane winning.
\end{theorem}

This theorem is quite similar to \cite[Theorem 15]{BadziahinHarrap}, with three main differences:
\begin{itemize}
\item We allow any sequence of rationals $\big(\frac{i_k}{j_k}\big)_k$, rather than just the sequence of all natural numbers.
\item We prove the theorem in any dimension, rather than just dimension 1.
\item We prove that the set in question is hyperplane winning, whereas Badziahin and Harrap show that it is Cantor-winning. These concepts share many similar properties and both imply full Hausdorff dimension, but the former is stronger \cite[Theorem 12]{BadziahinHarrap} and implies strong $C^1$ incompressibility on sufficiently regular fractals (see \cite{BFKRW}).
\end{itemize}

The special case $\frac{i_k}{j_k} = k$, $g(k) = k^{1/d}$ yields an answer to the latter part of Question \ref{questionbugeaud}, namely it shows that the set of badly approximable points which are not decaying is hyperplane winning, and therefore of full Hausdorff dimension.

We remark that this theorem also answers a related question. Namely, fix an irrational $i/j\in\Q$ and consider the sequence $i_k/j_k = (i/j)^k$. By \eqref{crude} we have $L\big((i/j)^k \xx\big) \ge \frac{1}{(i^{1/d}j)^k}L(\xx)$, so it is natural to ask about the size of the set of $\xx$ for which 
\begin{equation}
\label{powers}
\lim_{k \to \infty} (i^{1/d}j)^kL\left((i/j)^k \xx\right) = \infty.
\end{equation}
Taking $i_k/j_k = (i/j)^k$ and any $g(k) \in o\big((i^{1/d}j)^k\big)$, Theorem \ref{theoremslow} implies that for a full-dimension set of $\xx \in \BA_d$ there exists a sequence $(k_n)_n$ such that $L\big((i/j)^{k_n} \xx\big) \ge \frac{1}{g(k_n)}$ for all $n$. Fix such an $\xx$; applying \eqref{crude} one then obtains $L\big((i/j)^{k} \xx\big) \ge \left(g(k_n)(i^{1/d}j)^{k-k_n}\right)^{-1}$ for all $k_n \leq k < k_{n+1}$. It follows easily from this that \eqref{powers} holds for this $\xx$. Thus, the set of $\xx$ satisfying \eqref{powers} has full Hausdorff dimension in $\R^d$. As a concrete example, there is a full-dimension set of $x \in \R$ such that $L(2^n x) \in \omega(2^{-n})$.

On the other hand, in dimension one we show that the set of $x$ for which $L\big(\frac ij x\big)$ decays quickly also has full dimension, providing a complement to both of the above results. Namely, we obtain full dimension for the set of $x$ for which $\lim_{k \to \infty} (i j)^k L\big((\frac ij)^k x\big) < \infty$, as well as full dimension of the set of decaying badly approximable numbers, via the following theorem:

\begin{theorem}
\label{theoremfast}
The set $\left\{x \in \BA_1 \colon L\left(\frac ij x\right) \le \frac{1}{ij}\text{ for all reduced }\frac ij \in \Q\right\}$ has full Hausdorff dimension in $\R$.
\end{theorem}

Since the set in question is contained in the set of decaying badly approximable numbers, Theorem \ref{theoremfast} answers the first part of Question \ref{questionbugeaud} by showing that the set of decaying badly approximable points has full dimension in $\R$.



\section{Games}
\label{sectiongames}

In this section we describe two variants of Schmidt's game which will be used in the following sections. Both games are played by two players, whom we will call Alice and Bob. The \emph{hyperplane absolute game} (or just \emph{hyperplane game}) was introduced in \cite{BFKRW} and involves a parameter $0 < \beta < 1/3$ and a target set $S \subset \R^d$. Bob begins by choosing a ball $B_0 = B(\xx_0, \rho_0) \in \R^d$. Inductively, if $B_i = B(\xx_i, \rho_i)$ have been chosen for $i = 1, \dots, k$, Alice chooses a hyperplane $\LL_{k + 1}$ and a number $0 < \epsilon_{k + 1} \leq \beta$, determining a set $A_{k + 1} = \thickvar\LL{\epsilon_{k + 1} \rho_{k}} = \{\xx\in\R^d : \dist(\xx,\LL) \leq \epsilon_{k + 1} \rho_k\}$. Bob then must choose a ball $B_{k + 1} = B(\xx_{k + 1}, \rho_{k + 1}) \subset B_k \setminus A_{k + 1}$ satisfying $\rho_{k + 1} = \beta \rho_k$.\Footnote{The original hyperplane game only requires Bob's radius to satisfy $\rho_{k + 1} \geq \beta \rho_k$ rather than equality, making the game discussed here the ``modified hyperplane game''. However, the hyperplane game is equivalent to the modified hyperplane game (this follows from appropriately modifying the proof of \cite[Proposition 4.5]{FSU4}), so this distinction is not important.} The sets $B_0 \supset B_1 \supset B_2 \supset \dots$ are closed and nested with radii tending to zero, so the intersection $B_\infty = \bigcap_{k=0}^\infty B_k$ is a singleton. The unique point $\xx$ such that $B_\infty = \{\xx\}$ is called the \emph{result} of the game. If $\xx\in S$, Alice is declared the winner; otherwise, Bob wins. If Alice has a strategy to win this game regardless of how Bob plays for each $0 < \beta < 1/3$, then $S$ is called \emph{hyperplane absolute winning}, or just \emph{hyperplane winning}. Hyperplane winning sets intersect any sufficiently regular fractal in a set of full dimension, and the class of hyperplane winning sets is closed under diffeomorphisms and countable intersections. Combining these three properties one obtains \emph{strong $C^1$-incompressibility} of hyperplane winning sets on ``nice'' fractals (see \cite{BFKRW} for precise definition and discussion):

\begin{theorem}[{\cite[Corollary 5.4]{BFKRW}}]
Let $S$ be a hyperplane winning subset of $\R^d$, let $K$ be the support of an absolutely decaying and Ahlfors regular measure, and let $f_i : U \to \R^d$ ($i\in\N$) be nonsingular $C^1$ maps. Then
\[\dim\left( \bigcap_{i = 1}^\infty f^{-i}(S) \cap K\right) = \dim(K).\]
In particular, $\dim(S) = d$.
\end{theorem}

Many self-similar fractals are supports of absolutely decaying and Ahlfors regular measures so we obtain for example that when $S$ is a hyperplane winning subset of $\R$ and $C$ is the Cantor ternary set, there is a set of dimension $\dim(C)$ consisting of real numbers $x\in C$ such that $x, x^2, x^3, \dots$ are all in $S$.

The \emph{Banach--Mazur--Schmidt game}, or \emph{BMS game}, was introduced in \cite{FRS}. It is played on a complete metric space $(X,d)$ with a target set $S$ and parameter $0 < \beta < 1$ being given. Bob again begins by choosing an arbitrary ball $B_0 = B(x_0, \rho_0)$. Now suppose $B_0, B_1, \dots, B_k$ have been chosen and write $B_i = B(x_i, \rho_i)$. Alice then chooses a ball $A_{k + 1} = B(y_{k + 1}, \beta \rho_k)$ satisfying $d(x_k, y_{k + 1}) + \beta \rho_k \le \rho_k$. (Note that this implies $A_{k + 1} \subset B_k$.) Bob then chooses $B_{k + 1} = B(x_{k + 1}, \rho_{k + 1})$ satisfying $d(x_{k + 1}, y_{k + 1}) + \rho_{k + 1} \le \beta \rho_k$. Here Bob is allowed to choose $0 < \rho_{k + 1} \le \rho_k$ arbitrarily. Again, the balls $B_i$ ($i\in\N$) are closed and nested with radii tending to zero, so the intersection $B_\infty = \bigcap_k B_k$ is a singleton. Again, the \emph{result} $x$ is the unique point such that $B_\infty = \{x\}$, and we call Alice the winner if and only if $x\in S$. If Alice has a strategy to win this game regardless of how Bob plays, we say that $S$ is \emph{$\beta$-BMS winning}. If $S$ is $\beta$-BMS winning for some $0 < \beta < 1$ then we say that $S$ is \emph{BMS winning}. Informally, the BMS game is a two-player game in which the first player (Bob) chooses balls according to the relatively lax rules of the Banach--Mazur game, while the second player (Alice) chooses according to the stricter rules of Schmidt's game. Accordingly, the class of BMS winning sets is very restrictive and has strong geometric properties. To make this precise, we need the following definition:

\begin{definition}
Given $\beta > 0$, a set $E \subset X$ is said to be \emph{uniformly $\beta$-porous} if there exists $r_0 > 0$ such that for every ball $B(x,r) \subset X$ with $r \le r_0$, there exists an open ball $B^{\circ}(y,\beta r) \subset B(x,r)$ such that $B^{\circ}(y,\beta r) \cap E = \varnothing$.
\end{definition}

\begin{theorem}[{\cite[Theorem 2.2]{FRS}}]
\label{BMSporous}
Let $(X,d)$ be a complete, separable metric space and fix $0 < \beta < 1$. Then a Borel set $S \subset X$ is $\beta$-BMS winning if and only if $X \setminus \bigcap_{i=1}^\infty S_i$ is a countable union of uniformly $\beta$-porous sets.
\end{theorem}

When $(X,d)$ is given additional structure, this theorem yields an estimate on the Hausdorff measure of a BMS winning set. Specifically, we consider metric spaces $(X,d)$ which support a measure of the following type:

\begin{definition}
A Borel measure on a metric space $(X,d)$ is called \emph{Ahlfors $s$-regular} if there exist $C > 1$ and $\rho_0 > 0$ such that for any $x \in X$ and $0 < \rho \leq \rho_0$ we have
\begin{equation}
\label{Ahlfors}
\frac{1}{C} \rho^s \le \mu(B(x,\rho)) \le C \rho^s.
\end{equation}
We say $\mu$ is Ahlfors regular if it is Ahlfors $s$-regular for some $s > 0$.
\end{definition}

If $(X,d)$ supports some Ahlfors $s$-regular Borel measure, then the $s$-dimensional Hausdorff measure on $X$ is Ahlfors $s$-regular as well \cite[Theorem 8.5]{MSU}. It follows that $\dim X = s$. On a metric space $(X,d)$ supporting such a measure, BMS winning sets are very large in the sense that their complement has positive codimension:

\begin{proposition}[{\cite[Corollary 2.5]{FRS}}]
\label{propositionBMScomplement}
If $X$ is Ahlfors $s$-regular and $S \subset X$ is BMS winning, then $\dim(X\setminus S) < s$.
\end{proposition}

It follows from this that any BMS winning set on a metric space supporting an Ahlfors $s$-regular measure must be conull with respect to the $s$-dimensional Hausdorff measure.

\section{Proof of Theorem \ref{theoremslow}}
We will need the following lemma due to Davenport (see \cite[Lemma 4]{KTV}) for a version which implies the one stated here):

\begin{lemma}[Simplex Lemma]
\label{simplex}
Let $d \in \N$, $r > 0$, and $\xx \in \R^d$. Then the rational points in $B(\xx, r)$ with denominator at most $(2d! r)^{-\frac{d}{d+1}}$ all lie in a single hyperplane.
\end{lemma}

\begin{proof}[Proof of Theorem \ref{theoremslow}]
Fix $0 < \beta < 1/3$. Let $\rho_0 > 0$ be the radius of $B_0$. Since $g(k)\to\infty$, after extracting a subsequence we can assume without loss of generality that $\beta^{2^k + 1}/(4d!) \geq g(k)^{-1/2}$ for all $k$. (Taking a subsequence makes the set \eqref{slow} smaller, so if we can show that \eqref{slow} is hyperplane winning for the subsequence, then it is hyperplane winning for the whole sequence as well.) Write $\epsilon_k = g(k)^{-1/2}$. Alice will play the game in such a way that she guarantees that if $\xx$ denotes the result of the game, i.e. the unique point so that $\bigcap_{j \ge 0} B_j = \{\xx\}$, then for each $k \in \N$, $\left\|\frac{i_k}{j_k} \xx - \frac{\pp}{q}\right\| \ge \epsilon_{k}q^{-1 + 1/d}$ for all sufficiently large $q$. She will do this by partitioning $\N$ into infinitely many arithmetic progressions $P_k = 2^{k-1} + 2^k \N_0$ ($k\in\N$) and using her turns corresponding to indices in $P_k$ to ensure that the condition holds for $k$. (This argument mimics the original proof of the countable intersection stability of Schmidt's winning property.) This ensures that $L\left(\frac{i_k}{j_k}\xx\right)$ is larger than $\epsilon_k$.

Fix $k \in \N$ and $\ell\in P_k$, and we will describe Alice's strategy for the turn corresponding to the index $\ell$. Let 
\[
Q_m = \left(4d!\frac{i_k}{j_k} \beta^\ell \rho_0\right)^{-\frac{d}{d+1}}
\]
and let
\[
\QQ_m = \left\{\frac{j_k\pp}{i_kq} \colon Q_{m-1} \le q < Q_m \right\}.
\]
Note that every (reduced) $\frac{j_k\pp}{i_kq}$ with $q$ sufficiently large is in some $\QQ_m$ (where ``sufficiently large'' depends on $k$). Since $\ell\in P_k$, we have $\ell = 2^{k-1} + 2^k m$ for some $m \geq 0$. Given $B_\ell = B_{2^{k-1} + 2^k m}$, we consider $\QQ_m \cap 2B_{\ell}$, where $2B_\ell$ denotes the ball with the same center as $B_\ell$ and twice the radius. Applying Lemma \ref{simplex} with $r = 2(i_k/j_k)\beta^\ell \rho_0$, we see that all rationals in a given ball of radius $r$ of the form $\frac{\pp}{q}$ with $Q_{m-1} \le q < Q_m$ lie on a single hyperplane. Now, the map $\xx \mapsto \frac{j_k}{i_k}\xx$ sends balls of radius $2(i_k/j_k)\beta^\ell \rho_0$ to balls of radius $2\beta^\ell \rho_0$ and also sends hyperplanes to hyperplanes, so it follows that $\QQ_m \cap 2B_\ell$ is contained in a single hyperplane $\LL$. Alice will choose $A_\ell$ to be $\thickvar{\LL}{\beta^{\ell + 1}\rho_0}$. Then for any $\xx \in A_\ell$ and any $\frac{j_k\pp}{i_k q} \in \QQ_m\cap 2 B_\ell$, we have
\begin{align*}
\left\| \xx - \frac{j_k\pp}{i_k q}\right\| \ge \beta^{\ell+1}\rho_0 
			&= \frac{\beta^{2^k + 1}}{i_k/j_k} (i_k/j_k)\beta^{2^{k-1} + 2^k (m-1)}\rho_0\\
					&= \frac{\beta^{2^k + 1}}{i_k/j_k} \frac{Q_{m-1}^{-(1 + 1/d)}}{4d!}
			\ge \frac{\beta^{2^k + 1}}{4d!} \frac{1}{i_k/j_k} q^{-(1 + 1/d)}\\
							&\ge \frac{\epsilon_{k}}{i_k/j_k} q^{-(1 + 1/d)}.
\end{align*}
On the other hand, for any $\frac{j_k\pp}{i_k q} \in \QQ_m\butnot 2 B_\ell$ we have
\[
\left\| \xx - \frac{j_k\pp}{i_k q}\right\| \geq \beta^\ell\rho_0 \geq \beta^{\ell + 1}\rho_0 \geq \frac{\epsilon_{k}}{i_k/j_k} q^{-(1 + 1/d)}.
\]
Thus, if Alice plays according to this strategy and $\xx$ denotes the result of the game, then for each $k \in \N$, $\left\|\frac{i_k}{j_k} \xx - \frac{\pp}{q}\right\| \ge \epsilon_{k} q^{-(1 + 1/d)}$ for all $(\pp,q) \in \Z^d\times \N$.
\end{proof}

\section{Proof of Theorem \ref{theoremfast}}
Let $F_M \subset \R$ be the set of real numbers with partial quotients at most $M$. Recall that the set of badly approximable numbers is $\bigcup_{M \in \N} F_M$. We will prove that our set has full dimension in $\R$ by showing that it has full dimension within each $F_M$. To this end we prove the following:

\begin{proposition}
\label{propositionBMSwinning}
For each integer $M \ge 2$ and each reduced $\frac ij \in \Q$,  the set $\{x \in F_M \colon L\left(\frac ijx\right) \le \frac{1}{ij}\}$
is a BMS winning set on $F_M$.
\end{proposition}

Before proving this, we show how it can be used to deduce Theorem \ref{theoremfast}.

\begin{proof}[Proof of Theorem \ref{theoremfast}]
Fix $M \in \N$. One can easily check that $F_M \cap [0, 1]$ is the limit set of the contracting family of self-conformal maps $\{f_i\}_{i=1}^M$ on $[a,b]$, where $f_i(x) = \frac{1}{i+x}$, $a = [0; M, 1, M, 1, \dots]$ and $b = [0; 1, M, 1, M, \dots]$, and that this system satisfies the open set condition. It follows (see \cite[Lemma 3.14]{MauldinUrbanski1}) that the $\delta_M$-dimensional Hausdorff measure $\HH^{\delta_M}$ is Ahlfors regular, where $\delta_M = \dim(F_M)$. By Proposition \ref{propositionBMSwinning} the set
\[
\left\{x \in F_M \colon L\left(\frac ij x\right) \le \frac{1}{ij}\text{ for all }i, j \in \N\right\}
\]
is a countable intersection of BMS winning sets, hence conull with respect to $\HH^{\delta_M}$ by Proposition \ref{propositionBMScomplement}. In particular,
\[\dim\left\{x \in \R \colon L\left(\frac ij x\right) \le \frac{1}{ij}\text{ for all }i, j \in \N\right\} \ge \delta_M.\]
Since $\dim\big(\bigcup_M F_M\big) = \dim(\BA_1) = 1$, we have $\lim_{M\to\infty} \delta_M = 1$, so the theorem follows.
\end{proof}

The main idea of the proof of Proposition \ref{propositionBMSwinning} is as follows: Call an approximation $p/q$ of $x$ ``good'' if $|x - p/q| < 1/q^2$. If $i/j\in\Q$ is fixed and $(jp)/(iq)$ is a good approximation of $x$, then $|x - (jp)/(iq)| < 1/(iq)^2$, and rearranging gives $|(i/j)x - p/q| < 1/(ijq^2)$. Thus, if $x$ has infinitely many good approximations whose denominator is a multiple of $i$ and whose numerator is a multiple of $j$, then $L\big(\frac ij x\big) \leq \frac{1}{ij}$. Alice will direct the play of the game to numbers $x$ with this property by using her turns to choose digits of the continued fraction expansion in such a way that infinitely many truncated expansions (i.e. convergents), which are good approximations to $x$, have numerator and denominator in the required sets. To show that this is possible we first prove the following lemma:
 
\begin{lemma}
\label{lemmachoices}
For any $i, j \in \N$ and $M \ge 2$, there exists $T = T(i, j, M) \in \N$ such that for any $n \in \N$ and any $a_1, a_2, \dots, a_n \in \N$, there exist $1\le t \le T$ and $a_{n+1}, \dots, a_{n+t} \in \{1, \dots, M\}$ such that $[0; a_1, \dots, a_{n+t}]$ is a rational with numerator
in $j \N$ and denominator in $i\N$.
\end{lemma}

\begin{proof}
Recall \cite[Theorem 1]{Khinchin_book} that if $x = [0;a_1,a_2,\dots]$, then the approximations obtained by truncating the continued fraction expansion of $x$, $\frac{p_n}{q_n} = [0; a_1, a_2, \dots, a_n]$, satisfy the recurrence relation
\begin{equation}
\label{recur}
\frac{p_{n}}{q_{n}} = \frac{a_n p_{n-1} + p_{n-2}}{a_n q_{n-1} + q_{n-2}}\cdot
\end{equation}
In other words,
\[\brmatrix{p_{n-1} & q_{n-1}\\ p_n & q_{n}} 
					= \brmatrix{0 & 1\\ 1 & a_n}\brmatrix{p_{n-2} & q_{n-2} \\ p_{n-1} & q_{n-1}}.\]
Now it is well-known (e.g. \cite[Theorem 2]{Khinchin_book}) that
\begin{equation}
\label{det}
\det \brmatrix{p_{n-1} & q_{n-1}\\ p_n & q_n} = (-1)^n.
\end{equation}
Fix $i$, $j$, and $M$, and let $a_1, \dots, a_n \in \{1, \dots, M\}$. We may assume $i$ and $j$ are coprime (i.e. that the fraction $i/j$ is reduced). We want to choose $a_{n+1}, a_{n+2}, \dots, a_{n+t}$ in such a way that
\[\brmatrix{p_{n+t-1} & q_{n+t-1}\\ p_{n+t} & q_{n+t}} = \brmatrix{0 & 1\\ 1 & a_{n+t}}\brmatrix{0 & 1\\ 1 & a_{n+t-1}}
	\cdots \brmatrix{0 & 1\\ 1 & a_{n+1}} \brmatrix{p_{n-1} & q_{n-1}\\ p_{n} & q_{n}} 
				\equiv \brmatrix{m_1 & m_2\\j & i} \text{ (mod $ij$)},\]
where $m_1, m_2 \in \Z/ij\Z$ satisfy $m_1 j + m_2 i \equiv (-1)^{n} \text{ (mod $ij$)}$ (such numbers exist since $\gcd(i,j) = 1$). It will suffice to show that the semigroup $H$ generated by elements of the form $g_a = \brmatrix{0 & 1\\1 & a}$ ($1 \le a \le M$) is equal to $G = \SL^\pm(2,\Z/ij\Z)$. (Since $G$ is finite, if $H = G$ then there exists $T$ such that every element of $G$ is the product of at most $T$ elements of the form $g_a$.) Since $G$ is finite, for any $g \in H$ there exist $1 \le k_1 < k_2$ such that $g^{k_1} = g^{k_2}$ and therefore $g^{-1} = g^{k_2 - k_1 -1} \in H$, so $H$ is equal to the \emph{group} generated by the elements $g_a$. In particular, we have
\[
\brmatrix{-a & 1\\\phantom{-}1 & 0} = \brmatrix{0 & 1\\1 & a}^{-1} = g_a^{-1} \in H
\]
for all $1 \le a \le M$.

Recall (e.g. \cite[p.139]{Newman}) that $\brmatrix{1 & 1\\ 0 & 1}$ and $\brmatrix{0 & 1 \\ -1 & \phantom{-}0}$ generate $\SL(2,\Z)$. Since
\begin{align*}
\brmatrix{1 & 1\\0 & 1} &=  \brmatrix{-1 & 1\\\phantom{-}1 & 0}\brmatrix{0 & 1\\1 & 2} \in H\\
\brmatrix{1 & 0\\1 & 1} &=  \brmatrix{0 & 1\\1 & 2}\brmatrix{-1 & 1\\ \phantom{-}1 & 0} \in H\\
\brmatrix{\phantom{-}0 & 1\\-1 & 0} 
      &= \brmatrix{1 & 1\\0 & 1}\brmatrix{\phantom{-}1 & 0\\-1 & 1}\brmatrix{1 & 1\\0 & 1}
	= \brmatrix{1 & 1\\0 & 1}\brmatrix{1 & 0\\1 & 1}^{-1} \brmatrix{1 & 1\\0 & 1} \in H,
\end{align*}
it follows that $\SL(2,\Z/ij\Z) \subset H$. Since $\det(g_a) = -1$ for all $a$, we have $G = \SL^\pm(2,\Z/ij\Z) = H$, which completes the proof.
\end{proof}

\begin{proof}[Proof of Proposition \ref{propositionBMSwinning}]
A modification of the argument of \cite[Theorem 1.2]{McMullen_absolute_winning} shows that the BMS game is invariant under quasisymmetric homeomorphisms, and in particular under the coding map $\pi:\{1,\ldots,M\}^\N\to F_M$ defined by the formula $\pi(a_1,a_2,\ldots) = [0;a_1,a_2,\ldots]$. When the BMS game is played on the space $\{1,\ldots,M\}^\N$, the gameplay is equivalent to the following: on each of Bob's moves, he chooses a word in the alphabet $E = \{1,\ldots,M\}$ which is an extension of Alice's most recent word, and on each of Alice's moves, she chooses a word in the alphabet $E$ which is an extension of Bob's most recent word by $T$ letters, where $T\in\N$ is fixed at the start of the game. The result of the game is the unique infinite word $\omega$ which extends all of these words, so Alice's goal is to make the number $x = \pi(\omega)$ satisfy $L\big(\frac ij x\big) \leq \frac{1}{ij}$. Now let $T$ be given by Lemma \ref{lemmachoices}, and let Alice's strategy be given as follows: if Bob's most recent word is the move $a_1,\ldots,a_n$, then Alice's next move will be the word $a_1,\ldots,a_{n + T}$, where $a_{n + 1},\ldots,a_{n + t}$ are given by Lemma \ref{lemmachoices}, and $a_{n + t + 1},\ldots,a_{n + T}$ are arbitrary. When Alice uses this strategy, each of her moves determines a new convergent of the continued fraction expansion of the final word whose numerator is in $j\N$ and whose denominator is in $i\N$.

Now, suppose $\frac{p_k}{q_k}$ is a convergent with $p_k \in j\N$ and $q_k \in i\N$, say $p_k = jp$ and $q_k = iq$. Then
\[\left|x - \frac{jp}{iq}\right| = \left|x - \frac{p_k}{q_k}\right|  
\le \left|\frac{p_{k + 1}}{q_{k + 1}} - \frac{p_k}{q_k}\right| = \frac{1}{q_{k + 1}q_k} \le \frac{1}{i^2 q^2}\]
and hence
\[\left| \frac ijx - \frac{p}{q}\right| = \frac ij\left| x - \frac{jp}{iq}\right| \le \frac{i}{ji^2q^2} = \frac{1}{ij}\frac{1}{q^2}\cdot\]
Since this holds for infinitely many fractions $\frac{p}{q}$, it follows that $L\left(\frac ij x\right) \le \frac{1}{ij}$.
\end{proof}

\bibliographystyle{amsplain}

\bibliography{bibliography}

\end{document}